\documentclass[11pt,oneside, reqno]{amsart}

\usepackage[colorlinks,linkcolor=blue,citecolor=blue, urlcolor=blue]{hyperref}
\usepackage{graphics}
\usepackage{enumitem}
\usepackage{marginnote}	
\usepackage{latexsym, amssymb, amsmath, amsthm}
\usepackage{mathrsfs}
\usepackage{tikz}
\usepackage{tikz-cd}
\usepackage[only,llbracket,rrbracket]{stmaryrd}
\usetikzlibrary{decorations.pathmorphing}

\tikzset{
	squiggly/.style={decorate, decoration={snake, amplitude=1mm, segment length=6mm}},
}

\newtheorem{theorem}{Theorem}[section]
\newtheorem{lemma}[theorem]{Lemma}
\newtheorem{proposition}[theorem]{Proposition}
\newtheorem{corollary}[theorem]{Corollary}

\theoremstyle{definition}\newtheorem{definition}[theorem]{Definition}
\theoremstyle{definition}
\theoremstyle{definition}\newtheorem{remark}[theorem]{Remark}

\def\ga{\mathfrak{a}}
\def\gb{\mathfrak{b}}

\def\gm{\mathfrak{m}}
\def\gp{\mathfrak{p}}
\def\gq{\mathfrak{q}}

\def \sl{\mathfrak{sl}}

\def \CA{\mathcal{A}}
\def \CB{\mathcal{B}}
\def \CC{\mathcal{C}}
\def \CD{\mathcal{D}}

\def \CJ{\mathcal{J}}

\def \CR{\mathcal{R}}
\def \CS{\mathcal{S}}

\def \CY{\mathcal{Y}}

\def \Prim{{\rm Prim}}
\def\Max{{\rm Max}}

\def \Spec{{\rm Spec}}
\def \ann{{\rm ann}}
\def \Frac{{\rm Frac}}

\def \mK{\Bbbk}
\def \Z{\mathbb{Z}}
\def \N{\mathbb{N}}

\def \D{\Delta}
\def \d{\delta}
\def \s{ \sigma }
\def \l{\lambda}

\makeatletter
\@namedef{subjclassname@2020}{%
	\textup{2020} Mathematics Subject Classification}
\makeatother

\begin{document}

	\author{Tao Lu}
	
	
	
	\address{School of Mathematical Science, Yangzhou University, Yangzhou, 225002, China}
	
	\email{taolu@yzu.edu.cn} 
	
	\begin{abstract} 
We study the Hopf algebra $A$ given by the bosonization of the super Jordan plane over a field of characteristic not 2. We determine its centre, classical quotient ring, prime and primitive ideals. Over an algebraically closed base field, we classify all simple $A$-modules. In characteristic zero, simple modules are either finite-dimensional, of dimension 1 or 2, or infinite-dimensional, arising from a localization of the first Weyl algebra. Over an algebraically closed field of characteristic $p>2$, every simple module is finite-dimensional, of dimension 1, 2, or $2p$, and we give explicit representatives for all isomorphism classes. Our approach relies on realizing a suitable localization of $A$ as a matrix algebra.
	\end{abstract}
	
	\title[The bosonization of the super Jordan plane]{Prime ideals and representations of the bosonization of the super Jordan plane}
	\maketitle
	

\section{Introduction}

Nichols algebras play a central role in the classification of pointed Hopf algebras, notably through the lifting method of Andruskiewitsch and Schneider \cite{Andruskiewitsch-Schneider}. Among Nichols algebras of finite Gelfand--Kirillov dimension, the Jordan plane and the super Jordan plane are two of the most fundamental examples of non-diagonal type and serve as important prototypes for the study of Nichols algebras with nontrivial braiding; see, for instance, \cite{Andruskiewitsch-Angiono-Heckenberger,Andruskiewitsch-Angiono-Heckenberger-1,Andruskiewitsch-Angiono-Heckenberger-2}.

 The Jordan plane is the algebra generated by $\mathbf{x}$ and $\mathbf{y}$ subject to the defining relation 
 \begin{equation*}
 \mathbf{y} \mathbf{x}-\mathbf{x}\mathbf{y}=-\frac{1}{2}\mathbf{x}^2.
 \end{equation*}
 It is a Noetherian domain of Gelfand--Kirillov dimension $2$ whose structure and representation theory have been well-studied in the literature; see, for example, \cite{Artin-Shelter,Iyudu,Shirikov}.

The super Jordan plane,  introduced in \cite{Andruskiewitsch-Angiono-Heckenberger}, is a Noetherian algebra of Gelfand--Kirillov dimension~$2$. Unlike the Jordan plane, it is not a domain, and its $\mathbb{Z}/2\mathbb{Z}$-graded structure leads naturally to super-ring-theoretic phenomena. Explicitly, 
the super Jordan plane $\CJ$ is the $\mK$-algebra generated by $x$ and $y$ subject to the defining relations
\begin{equation*}
x^2=0,\qquad y^2x -xy^2-xyx=0. 
\end{equation*}
Its structural, homological, and representation-theoretic properties have attracted increasing attention in recent years.
In characteristic zero, finite-dimensional simple modules and small indecomposable modules were classified in \cite{Andruskiewitsch-Bagio-Della Flora-Flores-1}, while Hochschild homology and cohomology were computed in \cite{Reca-Solotar}. Moreover, \cite{Andruskiewitsch-Dumas} showed that $\CJ$, viewed as a $\mathbb{Z}/2\mathbb{Z}$-graded algebra, is super-prime and admits a super-simple super-Artinian ring of fractions; the groups of superalgebra and braided Hopf algebra automorphisms were also determined. Further structural and representation-theoretic properties, including the prime spectrum and classical ring of quotients, were studied in \cite{Lu}.

Let $G$ be the infinite cyclic group generated by $g$. The Nichols algebra $\CJ$ can be realized in the Yetter--Drinfeld category $_{\mK G}^{\mK G}\CY\CD$, giving rise to the Hopf algebra $A=\CJ\# \mK G$, where $\#$ denotes the Radford--Majid bosonization. Note that $A$ is a pointed Hopf algebra of Gelfand--Kirillov dimension 3.  Finite-dimensional simple $A$-modules were classified in characteristic zero in \cite{Andruskiewitsch-Bagio-Della Flora-Flores-2}, and the Drinfeld double of $A$ was studied in \cite{Andruskiewitsch-Pena Pollastri}.
Liftings of $A$ over nilpotent-by-finite groups were computed in \cite{Andruskiewitsch-Angiono-Heckenberger-1}. 

In this paper, we provide a detailed structural description of $A$, including its prime spectrum, classical quotient ring, and representation theory, over fields of characteristic different from 2. Our first main result determines the centre of $A$. More precisely, we show that $Z(A) = \mK[C]$ in characteristic zero, while in characteristic $p>2$ one has $Z(A) \cong \mK[X,Y,G^{\pm 1}, C]/\langle X^2G^{-1}-C^2 \rangle$.  We further obtain a complete classification of the prime, primitive, and maximal ideals of $A$. As a consequence, we obtain a classification of all simple $A$-modules. In characteristic zero, assuming $\mK$ is algebraically closed,  finite-dimensional simple modules have dimension 1 or 2, while infinite-dimensional simple modules arise via Morita equivalence from simple modules over a localization of the first Weyl algebra. The central element $C$ acts as zero on finite-dimensional simple modules and as a nonzero scalar on infinite-dimensional ones. If $\mK$ is algebraically closed of characteristic $p>2$, all simple $A$-modules are finite-dimensional, of dimension $1$, $2$, or $2p$, and we provide explicit constructions for each.

 We also determine the classical quotient ring of $A$.  In characteristic zero, $Q(A)  \cong M_2(\mathbf{D})$, where $\mathbf{D}$ is the skewfield of fractions of $\mK[C]\otimes A_1$, and $A_1$ denotes the first Weyl algebra. In characteristic $p>2$,  $Q(A)$ is a central simple algebra of dimension $(2p)^2$ over the centre $\mK(C, s^p, y^{2p})$, and $A$ is a prime PI algebra of PI degree~$2p$.

A key structural ingredient underlying these results is a matrix realization of a natural localization of $A$. Namely, if $s:=xy+yx$, we show that the localization $A_s$ of $A$ at the powers of $s$ is isomorphic to a matrix algebra. Although $A_s$ is not a domain, it is Morita equivalent to one. This realization allows many properties of both $A_s$ and $A$ to be deduced from classical results on Weyl algebras and matrix rings. In particular, in characteristic $p$, $A_s$ is an Azumaya algebra of rank $(2p)^2$ over its centre.

The paper is organized as follows. Section \ref{Localization} recalls basic properties of $A$ and studies the localization $A_s$. Section \ref{SupJordan-zero} treats the characteristic zero case, determining the centre, classical quotient ring, prime and primitive ideals, and simple modules. Section \ref{SupJordan-p} covers the positive characteristic case, describing the centre, PI degree, classical quotient ring, prime ideals, Azumaya locus, and simple modules.
 
 Throughout the paper, all modules are left modules. The base field $\mK$ is assumed to have characteristic different from~$2$, and additional assumptions will be imposed when needed. We write $\mK^{\times}=\mK\setminus{0}$ and $\mathbb{N}=\{0,1,2,\dots\}$. The centre of a ring $R$ is denoted by $Z(R)$.

\section{A localization of the bosonization of the super Jordan plane}  \label{Localization}  
In this section, we study the localization $A_s$ of the bosonization $A$ of the super Jordan plane at the powers of the normal element $s$. We also classify the prime ideals of the factor algebra $A/\langle s,x \rangle$.

\subsection{The bosonization of the super Jordan plane}
 The \emph{super Jordan plane} $\CJ$ is the $\mK$-algebra generated by $x$  and $y$ subject to the defining relations
\begin{equation} \label{SuperJordan} 
x^2=0, \qquad y s-sy-xs=0, \quad {\rm where}\,\, s:=xy+yx. 
\end{equation}

The algebra $\CJ$ is naturally $\N$-graded with $\deg\,x=\deg\,y=1$. By \cite{Andruskiewitsch-Angiono-Heckenberger},  it has Gelfand--Kirillov dimension 2, and the elements $x^i s^j y^k$ ($i=0,1$; $j,k\in \N$) form a PBW basis of $\CJ$. Furthermore, $\CJ$ is Noetherian of infinite global dimension by \cite[Corollary 1.4]{Andruskiewitsch-Dumas} and \cite{Reca-Solotar}, and it admits a presentation as an Ore extension by \cite[Proposition 1.3]{Andruskiewitsch-Dumas}.

Let $G=\langle g\rangle$ be the infinite cyclic group with generator $g$. Consider the Yetter--Drinfeld module $V=\mK x\oplus \mK y$ in $_{\mK G}^{\mK G}\CY\CD$ with grading $V_g=V$ and action given by $g\cdot x=-x$ and $g\cdot y=-y+x$. Then $\CJ$ is the Nichols algebra $\CB(V)$ (see \cite[Proposition 3.5]{Andruskiewitsch-Angiono-Heckenberger}). In particular, $\CJ$ is a Hopf algebra in the category $_{\mK G}^{\mK G}\CY\CD$, and its Radford--Majid bosonization $A=\CJ\#\mK G$ is a Hopf algebra in the usual category of $\mK$-vector spaces. We recall next an explicit presentation of $A$.

\begin{definition}
The \emph{bosonization of the super Jordan plane} is the $\mK$-algebra $A$ generated by $x,y,g,g^{-1}$ subject to the relations \eqref{SuperJordan} together with
\begin{equation} \label{bosonization} 
gg^{-1}=g^{-1}g=1, \qquad gxg^{-1}=-x, \qquad gyg^{-1}=-y+x. 
\end{equation}
Moreover, $A$ admits a Hopf algebra structure with comultiplication given by
\begin{equation*}
\D(g^{\pm 1})=g^{\pm 1} \otimes g^{\pm 1}, \quad \D(x)=x\otimes 1+ g\otimes x, \quad \D(y)=y \otimes 1+ g\otimes y. 
\end{equation*}
\end{definition}

From the defining relations \eqref{SuperJordan} and \eqref{bosonization}, one easily deduces
\begin{equation*}
sx=xs, \qquad gs=sg, \qquad sy^2-y^2 s=-s^2,  \qquad g y^2 g^{-1}=y^2-s. 
\end{equation*}
In particular, the third identity shows that the subalgebra of $A$ generated by $s$ and $y^2$ is isomorphic to the Jordan plane. Furthermore, $A$ can be written as a skew Laurent polynomial algebra over $\CJ$,
\begin{equation} \label{SkewA} 
A=\CJ[g^{\pm 1}; \s],
\end{equation}
where $\sigma$ is the automorphism of $\CJ$ determined by $\sigma(x)=-x$ and $\sigma(y)=-y+x$. In particular, $A$ is Noetherian of Gelfand--Kirillov dimension $3$,  and the set $\{ x^i s^j y^k g^{\ell} \mid i=0,1; j,k\in \N; \ell \in \Z \}$ forms a PBW basis of $A$.

\subsection{The localization $A_s$}
  
  The first Weyl algebra $A_1$ is the associative $\mK$-algebra generated by $\partial$, $x$ subject to the defining relation $\partial x-x \partial=1$. If ${\rm char}\,\mK=0$, then $A_1$ is a simple Noetherian domain with centre $Z(A_1)=\mK$.  If ${\rm char}\,\mK=p>0$, then $Z(A_1)=\mK[\partial^p, x^p]$, and $A_1$ is an Azumaya algebra over its centre. If moreover $\mK$ is algebraically closed of characteristic $p>0$, then every primitive ideal of $A_1$ is of the form $\langle \partial^p-\alpha, x^p-\beta \rangle$ for unique $\alpha,\beta\in \mK$, and the factor algebra $A_1/\langle \partial^p-\alpha, x^p-\beta \rangle$ is isomorphic to the matrix algebra $M_p(\mK)$, and so every simple $A_1$-module is $p$-dimensional (see \cite{Tsuchimoto}). 
  
  Recall that an element $a$ in a ring $R$ is said to be \emph{normal} if $aR=Ra$. From the relations $sx=xs$, $sy=(y-x)s$, and $sg=gs$, it follows immediately that $s$ is normal in both $\CJ$ and $A$.
  Let $\CJ_s$ (resp. $A_s$) denote the localization of $\CJ$ (resp. $A$) at the powers of $s$.
  By \eqref{SkewA}, the localization $A_s$ inherits a skew Laurent extension structure
  \begin{equation} \label{AsJs} 
  A_s=\CJ_s[g^{\pm 1};\s].
  \end{equation}

Next, we introduce the element $y':=ys^{-1}$. A straightforward computation shows that
\begin{equation*}
xy'+y'x=1, \quad y's-sy'=x, \quad gy'g^{-1}=-y'+xs^{-1}. 
\end{equation*}
It follows that 
\begin{equation*}
[y^{\prime 2},s]=[y',s]y'+y'[y',s]=xy'+y'x=1,
\end{equation*}
and therefore the subalgebra of $\CJ_s$ generated by $y^{\prime 2}$ and $s$ is isomorphic to the first Weyl algebra. 

A set of $n\times n$ \emph{matrix units} in a ring $T$ is a collection $\{e_{ij} \mid 1 \leq i,j \leq n \} \subset T$ satisfying 
$\sum_{i=1}^n e_{ii}=1$, and $e_{ij}e_{k\ell}=\delta_{jk}e_{i \ell}$
for all $i,j,k,\ell$. The existence of such a set of matrix
units provides a realization of $T$ as a full matrix algebra over a suitable
coefficient ring.

Set $e:=y'x \in \CJ_s$. Then $e^2=e$, so $e$ is an idempotent.  
Define
\begin{equation*}
e_{11}:=e, \quad e_{12}:=ey', \quad e_{21}:=x, \quad e_{22}:=1-e. 
\end{equation*}  
A direct calculation shows that $\{e_{ij}\mid 1 \leq i, j \leq 2\}$ is a set of matrix units in $\CJ_s$. Thus $\CJ_s$ admits a realization
as a $2\times2$ matrix algebra. The following lemma, which is
\cite[Lemma 2.1]{Lu}, identifies this matrix algebra structure explicitly.
\begin{lemma} \label{12Jan26}  
	{\rm(}\cite[Lemma 2.1]{Lu}{\rm).} 	Let $\CA$ be the subalgebra of $\CJ_s$ generated by $y^{\prime 2}$ and $s^{\pm 1}$. Then $\CA$ is a localization of the first Weyl algebra and 
	\begin{equation*}
	\CJ_s \cong M_2(\CA).
	\end{equation*}
	 More precisely, the map 
	\begin{equation*}
	\varphi: \CJ_s \longrightarrow M_2(\CA), \quad x \longmapsto \begin{bmatrix}
	0 & 0 \\ 1 & 0
	\end{bmatrix}, \quad s^{\pm 1} \longmapsto \begin{bmatrix}
	s^{\pm 1} & 0 \\ 0 & s^{\pm 1}
	\end{bmatrix}, \quad y' \longmapsto \begin{bmatrix}
	0 & 1 \\ y^{\prime 2} & 0
	\end{bmatrix},
	\end{equation*}
	defines an algebra isomorphism.
\end{lemma}

Define the element
\begin{equation*}
C:=(2yx-s)g^{-1}. 
\end{equation*}
A direct computation shows that $(2yx-s)^2=s^2$, and hence
\begin{equation*}
C^2=s^2g^{-2}.  
\end{equation*}
It follows that both $2yx-s$ and $C$ are invertible in $A_s$. Moreover, one verifies that $C$ commutes with the generators $x$, $y$, $g$, and therefore $C\in Z(A)$.

\begin{lemma} \label{7Feb26}  
	There are algebra isomorphisms
	\begin{equation*}
A_s \cong \mK[C^{\pm 1}] \otimes \CJ_s\cong\mK[C^{\pm 1}] \otimes M_2(\CA) \cong M_2(\mK[C^{\pm 1}]\otimes \CA). 
	\end{equation*}
\end{lemma}
\begin{proof}
	By \eqref{AsJs}, $A_s=\CJ_s[g^{\pm 1};\s]$ is a skew Laurent extension. Hence
	\begin{equation} \label{AsZ}
	A_s=\bigoplus_{i\in \Z} \CJ_s g^i.  
	\end{equation}
	 Since $g^{-1}=(2yx-s)s^{-2}C$, the generators $g^{\pm1}$ lie in the subalgebra generated by $\CJ_s$ and the central unit $C$. It follows that $A_s$ is generated by $\CJ_s$ and $C^{\pm1}$. 
	 
	 For every $i\in \Z$, we have 
	 \begin{equation*}
	 C^i=\begin{cases}
	 s^i g^{-i}, & i \text{ even},\\
	(2yx-s)s^{i-1}g^{-i}, & i \text{ odd}. 
	 \end{cases}
	 \end{equation*}
	 Therefore $C^i$ is homogeneous of $g$-degree $-i$. Suppose that $\sum_i a_i C^i=0$ with $a_i\in \CJ_s$.  By the direct sum decomposition \eqref{AsZ}, each homogeneous component
	 vanishes. Hence $a_i C^i=0$ for every $i$. Since $C$ is invertible in $A_s$, we obtain $a_i=0$ for all $i$. Therefore the elements $\{C^i\mid i\in\mathbb Z\}$
	 are linearly independent over $\CJ_s$, and consequently $A_s=\bigoplus_{i\in\mathbb Z}\CJ_sC^i.$ Since $C$ is central, this gives $A_s \cong \mK[C^{\pm 1}] \otimes \CJ_s$.  
	 The second isomorphism follows from Lemma~\ref{12Jan26}, and the last isomorphism is standard. 
\end{proof}



\subsection{Prime ideals of $A/\langle  s, x \rangle$}
Recall that a proper ideal $\gp$ of a ring $R$ is \emph{prime} if,  for all ideals $\ga$ and $\gb$ of $R$, the inclusion $\ga \gb \subseteq \gp$ implies $\ga \subseteq \gp$ or $\gb \subseteq \gp$. The set of prime ideals of $R$ is called the prime spectrum of $R$, and is denoted by $\Spec(R)$. 
A prime ideal $\gp$ is \emph{completely prime} if $R/\gp$ is a domain. For an $R$-module $M$, the annihilator of $M$ is the ideal of $R$ defined by
\begin{equation*}
\ann_R(M):=\{ r\in R \mid rm=0 \,\,\text{for all } m\in M \}. 
\end{equation*}
An ideal $\gp$ of a ring $R$ is \emph{primitive} if it is the annihilator of some simple $R$-module. The set of primitive ideals of $R$ is called the primitive spectrum of $R$, and is denoted by $\Prim(R)$.  It is well known that all primitive ideals are prime, and all maximal ideals are primitive. For a subset $\CS \subset R$, we write $\langle \CS\rangle$ for the two-sided ideal generated by $\CS$.

\begin{lemma} \label{11Mar26}  
	Let $P$ be a prime ideal of $A$. If $s\in P$, then $x\in P$. 
\end{lemma}
\begin{proof}
	Suppose $s\in P$.  Consider the factor algebra $\overline{A}:=A/sA$ and let $\overline{P}:=P/sA$, which is a prime ideal of $\overline{A}$.  From the defining relations we obtain $xy\equiv -yx \pmod {sA}$ and $xg\equiv-gx \pmod{sA}$, so the image $\bar{x}$ of $x$ in $\overline{A}$ is a normal element. Since $\bar{x}^2=0$, we have $(\bar{x}\overline{A})^2=\bar{x}^2 \overline{A}=0 \subseteq \overline{P}$. Since $\overline{P}$ is prime, it follows that $\bar{x}\overline{A} \subseteq \overline{P}$, and therefore $\bar{x}\in \overline{P}$. Consequently, $x\in P$. 
\end{proof}

Since $s=xy+yx\in \langle x \rangle$, one has $\langle s, x\rangle=\langle x\rangle$. Together with Lemma~\ref{11Mar26}, this shows that a prime ideal contains $s$ if and only if it contains $x$. 

Define $\Lambda:=A/\langle s, x \rangle$. Then 
\begin{equation*}
\Lambda \cong \mK \langle y, g^{\pm 1}\mid   gy=-yg,\,\, gg^{-1}=g^{-1}g=1 \rangle. 
\end{equation*}
The algebra $\Lambda$ can be presented as a skew Laurent polynomial algebra $\Lambda=\mK[y][g^{\pm 1};\s]$ where $\s(y)=-y$. In particular, $\Lambda$ is a domain, and the ideal $\langle s, \, x \rangle$ of $A$ is completely prime. 
Moreover, $\Lambda$ is a Hopf algebra with comultiplication determined by $\D(g)=g\otimes g$ and $\D(y)=y\otimes1+g\otimes y$. In fact, $\Lambda$ coincides with the quantum Borel subalgebra of $U_q(\sl_2)$ at $q=-1$. Note that $y^2$ and $g^2$ are central in $\Lambda$, and in fact, $Z(\Lambda)=\mK[y^2, g^{\pm 2}]$.  Furthermore, $\Lambda$ is a free $Z(\Lambda)$-module of rank 4 with
\begin{equation} \label{Lambda} 
\Lambda=\oplus_{i,j=0}^1 Z(\Lambda)y^i g^j. 
\end{equation}
 Thus $\Lambda$ is a Noetherian PI algebra.

 The following proposition describes the prime, primitive, and maximal spectra of $\Lambda$. 
 
 \begin{proposition}  \label{b8Feb26}   
 	Assume that $\mK$ is algebraically closed of characteristic $p\neq 2$. 
 	\begin{enumerate}
 		\item The prime spectrum of $\Lambda$ is 
 		\begin{equation*}
 		\begin{aligned}
 		\Spec(\Lambda)=\{ \langle 0 \rangle, \,\, \langle y \rangle \} \,\,\cup \,\,\{ \langle y,\, g-\gamma \mid \gamma \in \mK^{\times} \}\,\, \cup \,\, \{ \langle y^2-\alpha, \,g^2-\beta \rangle \mid \alpha, \beta\in \mK^{\times} \}\\\cup\,\, \{ \gq \Lambda \mid \gq \in \Spec(\mK[y^2, g^{\pm 2}]), \, {\rm ht}(\gq)=1, \, \gq \neq \langle y^2 \rangle \}. 
 		\end{aligned}
 		\end{equation*}
 		
 		\item The primitive spectrum of $\Lambda$ coincides with the maximal spectrum and is given by
 		\begin{equation*}
 		\Prim(\Lambda)=\Max(\Lambda)=\{ \langle y,\, g-\gamma \mid \gamma \in \mK^{\times} \}\,\, \cup \,\, \{ \langle y^2-\alpha, \,g^2-\beta \rangle \mid \alpha, \beta\in \mK^{\times} \}. 
 		\end{equation*}
 		The corresponding primitive factors are as follows:
 		\begin{enumerate}
 			\item For any $\gamma \in \mK^{\times}$, $\Lambda/\langle y, \, g-\gamma \rangle \cong \mK$. 
 			\item For any $\alpha, \beta \in \mK^{\times}$, $\Lambda/\langle y^2-\alpha, \,g^2-\beta \rangle \cong M_2(\mK)$. 
 		\end{enumerate}
 	\end{enumerate}
 \end{proposition}
 
 \begin{proof}
 	(1) Let $\Spec(\Lambda, y)$ (resp. $\Spec_y(\Lambda)$) denote the set of prime ideals of $\Lambda$ that contain (resp. do not contain) $y$. Then $\Spec(\Lambda)=\Spec(\Lambda, y) \,\cup \,\Spec_y(\Lambda)$.
 	
 	Prime ideals containing $y$ correspond bijectively to prime ideals of $\Lambda/y \Lambda \cong \mK[g^{\pm 1}]$. Since $\mK$ is algebraically closed, $\Spec(\mK[g^{\pm 1}])=\{ \langle 0 \rangle \}\,\cup \, \{ \langle g-\gamma \rangle \mid \gamma \in \mK^{\times} \}$. Lifting these ideals to $\Lambda$ gives $\Spec(\Lambda, y)=\{ \langle y \rangle \}\,\cup \, \{ \langle y, \, g-\gamma \rangle \mid \gamma \in \mK^{\times} \}$.
 	
 	Since $y$ is normal in $\Lambda$, prime ideals not containing $y$ correspond bijectively to prime ideals of the localization $\Lambda_y=\mK[y^{\pm 1}][g^{\pm 1};\s]$ where $\s(y)=-y$. The centre of $\Lambda_y$ is $Z(\Lambda_y)=\mathscr{Z}:=\mK[y^{\pm 2}, g^{\pm 2}]$, and $\Lambda_y$ is a free $\mathscr{Z}$-module of rank 4 with decomposition
 	\begin{equation} \label{Lambday} 
 	\Lambda_y=\oplus_{i,j=0}^1 \mathscr{Z}y^i g^j. 
 	\end{equation}
 	 By \cite[Proposition 7.2]{DeConcini-Procesi}, $\Lambda_y$ is an Azumaya algebra. Therefore, by \cite[Proposition 13.7.9]{MR}, extension and contraction induce a bijection
 	  $\Spec(\mathscr{Z}) \rightarrow \Spec(\Lambda_y)$ given by $\gp \mapsto \gp \Lambda_y$.  It follows that 
 	\begin{equation*}
 	\Spec_y(\Lambda)=\{ \gp \Lambda_y \cap \Lambda \mid \gp \in \Spec(\mathscr{Z}) \}.
 	\end{equation*} 	
 	The zero ideal of $\mathscr{Z}$ clearly contracts to the zero ideal $\langle 0 \rangle$ of $\Lambda$. The maximal ideals of $\mathscr{Z}$ are $\gm_{\alpha,\beta}=(y^2-\alpha, g^2-\beta)\mathscr{Z}$ where $\alpha, \beta\in \mK^{\times}$. We claim that $\gm_{\alpha,\beta} \Lambda_y \cap \Lambda=\langle y^2-\alpha, g^2-\beta \rangle$. Since the inclusion $\langle y^2-\alpha, g^2-\beta \rangle \subseteq \gm_{\alpha,\beta} \Lambda_y \cap \Lambda$ is clear,  it suffices to show that $\langle y^2-\alpha, g^2-\beta \rangle$ is maximal in $\Lambda$. Consider the homomorphism
 	\begin{equation*}
 	\Lambda \longrightarrow M_2(\mK), \qquad y \mapsto \begin{bmatrix}  \sqrt{\alpha} & 0 \\ 0 & -\sqrt{\alpha} \end{bmatrix}, \quad g \mapsto \begin{bmatrix}
 	0 & \beta  \\ 1 & 0
 	\end{bmatrix}.
 	\end{equation*}
 	Its kernel contains $\langle y^2-\alpha,g^2-\beta\rangle$.  Since the images generate $M_2(\mK)$, the map is surjective and therefore $\Lambda/\langle y^2-\alpha, g^2-\beta \rangle \cong M_2(\mK)$. Thus $\langle y^2-\alpha, g^2-\beta \rangle$ is maximal, proving the claim.   	
 	Finally, if $\gp$ is a height-one prime of $\mathscr Z$, then using \eqref{Lambda} and \eqref{Lambday} one verifies that
 	$\gp \Lambda_y \cap \Lambda=\gq \Lambda$ where $\gq=\gp \cap Z(\Lambda)$. Collecting all cases gives:
 	\begin{equation*}
 	\begin{aligned}
 	\Spec_y(\Lambda)
 = \{ \langle 0 \rangle \}\,\, &\cup \,\, \{ \langle y^2-\alpha, \,g^2-\beta \rangle \mid \alpha, \beta\in \mK^{\times} \}\\&\cup\,\, \{ \gq \Lambda \mid \gq \in \Spec(\mK[y^2, g^{\pm 2}]), \, {\rm ht}(\gq)=1, \, \gq \neq \langle y^2 \rangle \}. 
 	\end{aligned}
 	\end{equation*} 	
 This completes the description of $\Spec(\Lambda)$. 
 
 	(2) Since $\Lambda$ is a PI algebra, Kaplansky's theorem (see \cite[Theorem 13.3.8]{MR}) implies primitive ideals coincide with maximal ideals. The description of maximal ideals and primitive factors follows from statement (1).
 \end{proof}

\section{The bosonization of the super Jordan plane in characteristic zero} \label{SupJordan-zero} 
Throughout this section, $\mK$ denotes a field of characteristic zero.
We study the bosonization of the super Jordan plane, with emphasis on its
centre, classical quotient ring, prime and primitive spectra, and simple modules.

\subsection{Centre and quotient ring of $A$}
Let $\CC_R$ denote the set of regular elements of a ring $R$. Then $\CC_R$ is a multiplicative set. If $\CC_R$ is a left (resp. right) Ore set, the localization $Q_{\ell}(R):=\CC_R^{-1}R$ (resp. $Q_{r}(R):=R \CC_R^{-1}$) is called the left (resp. right) quotient ring of $R$. If $\CC_R$ is an Ore set (that is, both a left and a right Ore set), then $Q_{\ell}(R)$ and $Q_r(R)$ are isomorphic, and the ring $Q(R):=\CC_R^{-1}R=R \CC_R^{-1}$ is called the \emph{classical quotient ring} of $R$.

 Since $A$ is Noetherian, the set of regular elements satisfies both the left and right Ore conditions. We now determine the centre and the classical quotient ring of $A$.
\begin{proposition}
	Assume ${\rm char}\,\mK=0$. 
	\begin{enumerate}
		\item The centre of $A$ is the polynomial algebra $Z(A)=\mK[C]$. 
		\item The classical quotient ring of $A$ is $Q(A)  \cong M_2(\mathbf{D})$, a central simple algebra over the Weyl skewfield $\mathbf{D}:=\Frac(\mK[C] \otimes A_1)$. 
	\end{enumerate}
\end{proposition}
\begin{proof}
	(1) Since ${\rm char}\,\mK=0$, the centre of $\CA$ is trivial, and hence so is the centre of $M_2(\CA)$. By Lemma~\ref{7Feb26}, it follows that $Z(A_s)=\mK[C^{\pm 1}]$. Therefore $Z(A) =A \cap Z(A_s)=A \cap \mK[C^{\pm 1}]$. We claim that  $A \cap \mK[C^{\pm 1}]=\mK[C]$. Suppose otherwise. Then $C^{-n}\in A$ for some $n\geq 1$. Replacing $n$ by $2n$ if necessary, we may assume that $n$ is even. Since $C^2=s^2 g^{-2}$, we obtain  $C^{-n}=s^{-n}g^n\in A$. It follows that $1\in Ag^{-n}s^n =As^n \subseteq As$, which implies $As=A$,  contradicting the fact that $As$ is a proper ideal of $A$. Therefore, $A \cap \mK[C^{\pm 1}]=\mK[C]$, and hence $Z(A)=\mK[C]$. 
	
	(2). By Lemma \ref{7Feb26}, we have $A_s \cong M_2(\CR)$ where $\CR=\mK[C^{\pm 1}] \otimes \CA$. Since localization preserves the classical quotient ring, by \cite[Corollary 3.1.5]{MR}, we obtain $Q(A)=Q(A_s) \cong M_2(Q(\CR))$. Since $\CR$ is a Noetherian domain, its classical quotient ring is its skewfield of fractions: $Q(\CR)=\Frac(\CR)\cong \mathbf{D}$. The result follows. 
\end{proof}

\subsection{Prime ideals of $A$}
For any $\l \in \mK$, define 
\begin{equation*}
A(\l):=A/(C-\l)A. 
\end{equation*}

\begin{lemma} \label{a8Feb26} 
	Assume ${\rm char}\,\mK=0$. Then the algebra $A(\l)$ is simple if and only if $\l \in \mK^{\times}$. 
\end{lemma}
\begin{proof}
	First consider the case $\l=0$. Since $C=(2yx-s)g^{-1}$, we have 
	 $CA\subseteq \langle x, \, s\rangle$. The ideal $\langle x, \,s\rangle$ is not maximal because the factor algebra $\Lambda=A/\langle x, \,s \rangle$ is not simple. Hence $CA$ is not maximal, and therefore $A(0)$ is not simple. 
	Now assume $\l \in \mK^{\times}$. Then 
	\begin{equation*}
	C^2=s^2g^{-2} \equiv \l^2 \pmod {(C-\l)A}. 
	\end{equation*}
    Hence, in $A(\lambda)$, we have $s^2=\lambda^2 g^2.$
	Since $\l \neq 0$ and $g$ is invertible, the element $s$ is invertible in $A(\l)$, with inverse $s^{-1}=\lambda^{-2}sg^{-2}$. Consequently, localization at the powers of $s$ does not change the algebra $A(\l)$,  and by Lemma~\ref{7Feb26} we obtain
	\begin{equation} \label{ACL} 
	A(\l)=A(\l)_s \cong A_s/(C-\l)A_s \cong M_2(\CA). 
	\end{equation}
	Since $\CA$ is simple, so is $M_2(\CA)$. Thus $A(\lambda)$ is simple for all $\lambda\neq 0$.
\end{proof}

The following theorem gives a complete description of the prime, primitive and maximal ideals of the algebra $A$ in characteristic zero.

\begin{theorem} \label{9Feb26}  
	Assume that $\mK$ is algebraically closed of characteristic zero. 
\begin{enumerate}
	\item The prime spectrum of $A$ is 
	\begin{equation*}
	\Spec(A)=\{ \langle 0 \rangle\} \,\,\cup\,\,\{ \langle s, \, x,\, \ga \rangle \mid \ga \in \Spec(\Lambda) \} \,\,\cup \,\, \{ (C-\l)A \mid \l \in \mK^{\times} \}. 		
	\end{equation*}
	
	\item The primitive spectrum of $A$ coincides with the maximal spectrum and is given by 
	\begin{equation*}
	\begin{aligned}
	\Prim(A)=\Max(A)= &\{ \langle s, \, x, \, y,\, g-\gamma \mid \gamma \in \mK^{\times} \}\\ &\cup \,\, \{ \langle s,\, x,\,  y^2-\alpha, \,g^2-\beta \rangle \mid \alpha, \beta\in \mK^{\times} \} \\ &\cup \,\, \{ (C-\l)A \mid \l \in \mK^{\times} \}. 
	\end{aligned}
	\end{equation*}
	The corresponding primitive factor algebras are as follows:
	\begin{enumerate}
		\item For any $\gamma \in \mK^{\times}$, $A/\langle s, \, x,\, y, \, g-\gamma \rangle \cong \mK$. 
		\item For any $\alpha, \beta \in \mK^{\times}$, $A/\langle s, \,x,\, y^2-\alpha, \,g^2-\beta \rangle \cong M_2(\mK)$. 
		\item For any $\l \in \mK^{\times}$, $A/(C-\l)A \cong M_2(\CA)$. 
	\end{enumerate}

		\end{enumerate}
\end{theorem}
\begin{proof}
	(1) Let $\Spec(A, s)$ (resp. $\Spec_s(A)$) denote the set of prime ideals of $A$ containing (resp. not containing) $s$. Then $\Spec(A)=\Spec(A, s) \cup \Spec_s(A)$. 
	
	By Lemma \ref{11Mar26}, every prime ideal of $A$ containing $s$ also contains $x$. Hence the prime ideals of $A$ containing $s$ are in bijection with the prime ideals of the factor algebra 
	 $\Lambda=A/\langle s, \, x\rangle$. In particular, $\Spec(A, s)=\{ \langle s, \, x,\, \ga \rangle \mid \ga \in \Spec(\Lambda) \}$. 
	
	Since $s$ is a normal element of $A$, prime ideals of $A$ not containing $s$ are in bijection with prime ideals of the localization $A_s$. By Lemma~\ref{7Feb26}, we have $A_s\cong \mK[C^{\pm 1}] \otimes M_2(\CA)$, where $M_2(\CA)$ is simple with centre $\mK$.  Therefore, prime ideals of $A_s$ correspond bijectively to prime ideals of $\mK[C^{\pm 1}]$, 
	\begin{equation*}
	\Spec(A_s)=\{\langle 0 \rangle\}\,\,\cup \,\,\{(C-\l)A_s \mid \l \in \mK^{\times}\}. 
	\end{equation*}
	The zero ideal of $A_s$ contracts to $\langle 0 \rangle$ in $A$. Moreover, for each $\l \in \mK^{\times}$, $A \cap (C-\l)A_s=(C-\l)A$, since $(C-\lambda)A$ is maximal by Lemma~\ref{a8Feb26}.  Therefore, $\Spec_s(A)=\{ \langle 0 \rangle \}\,\, \cup \,\,\{ (C-\l)A \mid \l \in \mK^{\times} \}$. Combining the two cases gives the stated description of $\Spec(A)$.

	(2) Since $A$ is Noetherian, every primitive ideal is prime. We determine which of those primes are primitive. 
	First consider the primes containing $s$. As shown above, these are precisely the ideals of the form $\langle s,\,x,\, \ga \rangle$ with $\ga \in \Spec(\Lambda)$. By Proposition~\ref{b8Feb26}(2), the primitive ideals of $\Lambda$ are known.    Lifting each $\ga \in \Prim(\Lambda)$ to $A$ gives
	\begin{equation*}
	\{ \langle s, \, x, \, y,\, g-\gamma \mid \gamma \in \mK^{\times} \}\,\, \cup \,\, \{ \langle s,\, x,\,  y^2-\alpha, \,g^2-\beta \rangle \mid \alpha, \beta\in \mK^{\times} \}. 
	\end{equation*}	
	Next consider the primes not containing $s$. For each $\l\in \mK^{\times}$, the ideal $(C-\l)A$ is maximal by Lemma~\ref{a8Feb26}, hence primitive. The zero ideal $\langle 0 \rangle$ is not primitive since $Z(A)=\mK[C]$.  
	
	Thus, all primitive ideals are maximal and the descriptions of their factor algebras follow from Proposition~\ref{b8Feb26}(2) and \eqref{ACL}.
\end{proof}


\begin{corollary}
	Assume that $\mK$ is algebraically closed of characteristic zero. Then
	the Jacobson radical of $A$ is zero.
\end{corollary}

\begin{proof}
	By definition,
	${\rm Jac}(A)=\bigcap_{P\in\operatorname{Prim}(A)}P$.
	Since the ideals $(C-\lambda)A$, $\lambda\in\mK^\times$, are primitive
	ideals, we have
	${\rm Jac}(A)\subseteq \bigcap_{\lambda\in\mK^\times}(C-\lambda)A$.
	Set $I:=\bigcap_{\lambda\in\mK^\times}(C-\lambda)A $.
	Suppose that $I\neq0$. For every $\lambda\in\mK^\times$,
	$(C-\lambda)A$ is a prime ideal containing $I$. By the description of
	$\operatorname{Spec}(A)$, the only prime ideal properly contained in
	$(C-\lambda)A$ is $\langle 0 \rangle$. Since $I\neq0$, it follows that
	$(C-\lambda)A$ is a minimal prime over $I$. Hence there are infinitely many prime
	ideals minimal over $I$, contradicting the fact that $A$ is Noetherian.
	Therefore $I=0,$ and consequently ${\rm Jac}(A)=0$.
\end{proof}

\subsection{Classification of simple $A$-modules}

We first construct a family of two-dimensional simple $A$-modules $T_{\alpha, \beta}$, parametrized by $\alpha, \beta \in \mK^{\times}$.
As a vector space $T_{\alpha, \beta}=\mK e_0 \oplus \mK e_1$. The $A$-action is given by
\begin{equation} \label{Tab} 
\begin{gathered} 
s e_i=0,  \quad x e_i=0 \quad  (i=0,1), \\
ye_0=\alpha e_0, \quad ye_1=-\alpha e_1,  \qquad g e_0=e_1, \quad g e_1=\beta e_0. 
\end{gathered}
\end{equation}
A direct verification shows that these formulas satisfy the defining relations of $A$, and hence determine an $A$-module structure on $T_{\alpha,\beta}$. Since $\langle s,x\rangle$ annihilates $T_{\alpha,\beta}$, this module factors through the quotient $\Lambda=A/\langle s,x\rangle$. 

\begin{lemma} \label{a26Mar26}  
	For any $\alpha,\beta \in \mK^{\times}$, the module $T_{\alpha,\beta}$ is simple. Its annihilator is 
	\begin{equation*}
	\ann_A(T_{\alpha,\beta})=\langle s, x, y^2-\alpha^2, g^2-\beta \rangle.
	\end{equation*}
   Furthermore,  $T_{\alpha,\beta} \cong T_{\alpha', \beta'}$  if and only if $\alpha^{\prime 2}=\alpha^2$ and $\beta'=\beta$. 
\end{lemma}
\begin{proof}
	Let $0\neq v=a e_0+b e_1$ be an element of a nonzero submodule of $T_{\alpha,\beta}$. If $a\neq 0$, then $e_0\in Av$, and applying $g$ gives $e_1\in Av$. Similarly, if $b\neq 0$, then $e_1\in Av$, and applying $g$ gives $e_0\in Av$. Thus $Av=T_{\alpha,\beta}$, proving simplicity.
	The ideal $I=\langle s, x, y^2-\alpha^2, g^2-\beta \rangle$ annihilates $T_{\alpha,\beta}$, so $I \subseteq \ann_A (T_{\alpha,\beta})$. By Theorem \ref{9Feb26}(2),  $I$ is maximal, hence 
	$\ann_A (T_{\alpha,\beta})=I$. 
	Finally, since $A/I \cong M_2(\mK)$ has a unique simple module up to isomorphism, two modules $T_{\alpha,\beta}$ and $T_{\alpha', \beta'}$ are isomorphic if and only if their annihilators coincide, which gives the stated conditions.
\end{proof}

For an algebra $R$, let $\widehat{R}$ denote the set of isomorphism classes $[M]$ of simple $R$-modules $M$. 
\begin{theorem} \label{11Feb26} 
	Assume that $\mK$ is algebraically closed of characteristic $0$.
	\begin{enumerate}
		\item The set of  isomorphism classes of finite-dimensional simple $A$-modules is 
		\begin{equation*}
		\widehat{A}({\rm fin. \,\,dim.})= \{ [S_{\gamma}] \mid \gamma \in \mK^{\times} \} \,\,\cup \,\, \{ [T_{\alpha,\beta}] \mid \alpha,\beta \in \mK^{\times} \}
		\end{equation*}
		where  $S_{\gamma}:=A/\langle s, \, x, \, y,\, g-\gamma \rangle$ is one-dimensional. 		
		
		\item A simple $A$-module $M$ is infinite-dimensional if and only if the central element $C$ acts on $M$ as a nonzero scalar. Thus
		\begin{equation*}
		\widehat{A}({\rm infin. \,\,dim.})=\bigsqcup_{\l \in \mK^{\times}} \widehat{A(\l)}.  
		\end{equation*}
		Furthermore, the category of $A(\l)$-modules is equivalent to the category of $\CA$-module. In particular, there is a bijection $\widehat{A(\lambda)}\leftrightarrow\widehat{\CA}$.
	\end{enumerate}
\end{theorem}
\begin{proof}
	(1) Let $M$ be a finite-dimensional simple $A$-module and set $P:=\ann_{A}(M)$.  Then $P$ is a primitive ideal of finite codimension. By Theorem \ref{9Feb26}(2), either $A/P \cong \mK$ or $A/P \cong M_2(\mK)$. Hence $\dim M=1$ or $2$. If $\dim M=1$, then $M$ is annihilated by $\langle s,x,y\rangle$, and $g$ acts as a scalar $\gamma\in\mK^{\times}$. Thus $M\cong S_{\gamma}$. Suppose now that $\dim M=2$. Then again by Theorem \ref{9Feb26}(2), $P=\langle s,\, x,\,y^2-\alpha^2, \, g^2-\beta \rangle$ for some $\alpha,\beta\in\mK^{\times}$. By Lemma \ref{a26Mar26}, $\ann_A(T_{\alpha,\beta})=P$. Thus both $M$ and $T_{\alpha,\beta}$ are simple modules over $A/P \cong M_2(\mK)$. Since this algebra has a unique simple module, it follows that $M \cong T_{\alpha,\beta}$.

	(2) Let $M$ be an infinite-dimensional simple $A$-module. By Theorem~\ref{9Feb26}(2), $\ann_A(M)=(C-\lambda)A$ for some $\lambda\in\mK^{\times}$. Thus $M$ is a simple module over $A(\lambda)=A/(C-\lambda)A$.	
	By \eqref{ACL}, $A(\lambda)\cong M_2(\CA)$, which is a simple infinite-dimensional algebra. In particular, every simple $A(\lambda)$-module is infinite-dimensional. Since matrix algebras are Morita equivalent to their coefficient algebras, the category of left $A(\lambda)$-modules is equivalent to the category of left $\CA$-modules. This equivalence preserves simple modules and yields a bijection $\widehat{A(\l)} \leftrightarrow \widehat{\CA}$. This completes the proof.
\end{proof}

\begin{remark}
	Assume $\mK$ is algebraically closed of characteristic $0$. Theorem \ref{11Feb26} provides a complete classification of simple $A$-modules. Namely, the simple $A$-modules consist of the finite-dimensional modules $S_\gamma$ and $T_{\alpha,\beta}$ together with the infinite-dimensional modules arising from simple $\CA$-modules.
	The action of the central element $C$ distinguishes these cases: $C$ acts as $0$ on finite-dimensional simple modules and as a nonzero scalar on infinite-dimensional ones. 
	
	More explicitly, for every $\lambda\in\mK^\times$, the isomorphism
	$\varPhi_{\lambda}:A(\lambda)\stackrel{\sim}{\longrightarrow}M_2(\CA)$
	provides a construction of the infinite-dimensional simple $A$-modules. By Lemma \ref{12Jan26}, Lemma \ref{7Feb26} and \eqref{ACL},  $\varPhi_{\l}$ is
	given by
		\begin{equation*}
	x \mapsto \left[\begin{smallmatrix}
	0 & 0 \\ 1 & 0
	\end{smallmatrix} \right], \quad s^{\pm 1} \mapsto \left[\begin{smallmatrix}
	s^{\pm 1} & 0 \\ 0 & s^{\pm 1}
	\end{smallmatrix}\right], \quad y \mapsto \left[\begin{smallmatrix}
	0 & s \\ y^{\prime 2}s & 0
	\end{smallmatrix} \right], \quad g \mapsto \left[\begin{smallmatrix}
	\l^{-1}s & 0 \\ 0  & -\l^{-1}s
	\end{smallmatrix} \right], \quad g^{-1} \mapsto \left[\begin{smallmatrix}
	\l s^{-1} & 0 \\ 0  & -\l s^{-1}
	\end{smallmatrix} \right].
	\end{equation*}
	 If $V$ is a simple $\CA$-module, then $V\oplus V$ is a simple $M_2(\CA)$-module with the action 
	\begin{equation*}
	\begin{bmatrix}
	a & b \\ c& d 
	\end{bmatrix} \begin{bmatrix}
	v_1 \\ v_2
	\end{bmatrix} =\begin{bmatrix}
	av_1+bv_2 \\ c v_1+ d v_2
	\end{bmatrix}, \qquad a, b,c,d \in \CA, \quad  v_1, v_2\in V. 
	\end{equation*}
	Transporting this module structure through $\varPhi_\lambda$, namely by
	setting
	\[
	a\cdot
	\begin{bmatrix}
	v_1\\v_2
	\end{bmatrix}
	:=
	\varPhi_\lambda(a)
	\begin{bmatrix}
	v_1\\v_2
	\end{bmatrix},
	\qquad a\in A(\lambda),
	\]
	gives a simple $A(\lambda)$-module. Via the quotient map
	$A\rightarrow A(\lambda)$, it yields a simple $A$-module on which the
	central element $C$ acts as the scalar $\lambda$.
	
	The reduction to simple $\CA$-modules connects the classification of simple
	$A$-modules with the study of simple modules over the first Weyl algebra and
	its Ore localizations, whose representation theory has been developed by Block
	\cite{Block} and Bavula \cite{Bavula-GWA,Bavula-OreD}.
	
\end{remark}

\section{The bosonization of the super Jordan plane in positive characteristic}  \label{SupJordan-p} 

In this section we study the bosonization of the super Jordan plane over a field of positive characteristic. We determine the centre, the PI degree, and the classical quotient ring, and describe the prime spectrum, the Azumaya locus, and the simple modules.

\subsection{Centre, PI degree, and classical quotient ring}
Recall that the \emph{PI degree} of a prime PI ring $R$, denoted $\operatorname{PI-deg} R$, is the integer $d$ such that $d^2$ is the dimension of the central simple quotient ring $Q(R)$ over its centre $Z(Q(R))$. 
If $R$ is moreover an affine $\mK$-algebra with $\mK$ algebraically closed, then $\operatorname{PI-deg} R$ equals the maximal dimension of simple $R$-modules. A prime ideal $\gp$ of a prime PI ring $R$ is \emph{regular} if $\operatorname{PI-deg} R/\gp=\operatorname{PI-deg} R$.

\begin{proposition} \label{23Jan26}  
	Assume ${\rm char}\,\mK=p >2$. Then the following statements hold:
	\begin{enumerate}
		\item The centre of $A_s$ is $Z(A_s)=\mK[C^{\pm 1}, s^{\pm p}, y^{2p}]$.
		\item The centre of $A$ is  generated by $X:=sg^{p-1}$, $Y:=y^{2p}$, $G^{\pm 1}=g^{\pm 2p}$, and $C$, subject to the relation $X^2G^{-1}=C^2$.  Thus $Z(A) \cong \mK[X,Y,G^{\pm 1}, C]/\langle X^2G^{-1}-C^2 \rangle$.
	
		\item $A_s$ is an Azumaya algebra of rank $(2p)^2$ over its centre. 
		\item $A_s$ is a prime PI ring whose prime ideals are all regular, and $\operatorname{PI-deg}A_s=2p$. 
 		\item There is a bijective correspondence between the ideals of $A_s$ and the ideals of $Z(A_s)$ given by $I \mapsto I \cap Z(A_s)$ with inverse $\ga \mapsto \ga A_s$. Moreover, this correspondence preserves primeness. 
		\item Every simple $A_s$-module has dimension $2p$. 
		\item $A$ is a prime PI ring with $\operatorname{PI-deg} A=2p$. 
		\item Let $Z:=Z(A)$, $\mathscr{Z}:=Z(A_s)$, and $K:=\Frac(Z)=\mK(C, s^p, y^{2p})$. Then 
		\begin{equation*}
		Q(A) \cong A \otimes_Z K \cong M_2(\CA)\otimes_{\mathscr{Z}} K, 
		\end{equation*}
		and $Q(A)$ is a central simple algebra of dimension $(2p)^2$ over $K$.
	\end{enumerate}
\end{proposition}
\begin{proof}
	(1) By \cite[Proposition 4.3(1)]{Lu}, the center of $\CJ_s$ is 
	\begin{equation*}
	Z(\CJ_s)=\mK[s^{\pm p}, y^{2p}].
	\end{equation*}
	Since $A_s \cong \mK[C^{\pm 1}] \otimes \CJ_s$ (Lemma~\ref{7Feb26}), it follows that $Z(A_s)=\mK[C^{\pm 1}, s^{\pm p}, y^{2p}]$.
	
	(2) First observe that $C^{p-1}=s^{p-1}g^{-p+1}$, and therefore $s^p=XC^{p-1}$.
	Since $s^p$ is central in $A_s$, it follows that $X$ is central in $A_s$ and hence $X\in Z(A)$.  From statement (1), we obtain $Z(A_s)=\mK[C^{\pm 1}, X^{\pm 1}, Y]$. Next observe that
	\begin{equation*}
	X^2G^{-1}=s^2 g^{2p-2}\cdot g^{-2p}=s^2g^{-2}=C^2. 
	\end{equation*}
	Thus $G^{\pm1}$ also belongs to $Z(A_s)$ and therefore lies in $Z(A)$.
	
	From the PBW basis of $A$, the elements $X$, $Y$, and $G^{\pm 1}$ are algebraically independent. Let $R:=\mK[X^{\pm 1}, Y, G^{\pm 1}]$. Then $C$ is integral over $R$ since it satisfies the monic polynomial $t^2-X^2G^{-1}=0.$
	
	We claim that $C\notin R$. Indeed, the skew Laurent extension
	$A_s=\CJ_s[g^{\pm1};\sigma]$
	has a $\mathbb Z$-grading induced by the powers of $g$, see \eqref{AsZ}. The generators of
	$R$ are homogeneous with
	\[
	\deg_g(X)=p-1,\qquad \deg_g(Y)=0,\qquad \deg_g(G)=2p.
	\]
	Since $p$ is odd, all these degrees are even. Hence every homogeneous
	component occurring in an element of $R$ has even $g$-degree. On the other
	hand, $C=(2yx-s)g^{-1}$ is homogeneous of $g$-degree $-1$, which is odd. Therefore $C\notin R$.
	
	Since $C^2=X^2G^{-1}\in R$ and $C\notin R$, the minimal polynomial of $C$
	over $R$ has degree two. Therefore $Z(A_s)=R[C]$ is a free $R$-module
	with basis $\{1,C\}$:
	\begin{equation} \label{Asdec} 
	Z(A_s)=R \oplus RC. 
	\end{equation}
	
	We claim that 
	\begin{equation*}
	R \cap A=\mK[X,Y,G^{\pm 1}], \qquad RC \cap A=\mK[X,Y,G^{\pm 1}]C. 
	\end{equation*}
	Suppose $R\cap A\neq \mK[X,Y,G^{\pm1}]$. Then there exists $u=\sum_{i=-n}^m f_i X^i \in R \cap A$ where $f_i \in \mK[Y,G^{\pm 1}]$, $n>0$, and $f_{-n} \neq 0$. Multiplying by $X^n$ gives $f_{-n} \in AX^n \subseteq As$.  However, since $f_{-n}\in\mK[Y,G^{\pm1}]$ involves only $y$ and $g$, the PBW basis of $A$ shows that $f_{-n}\notin As$, a contradiction. 	
	Similarly, if $RC\cap A\neq \mK[X,Y,G^{\pm1}]C$,
	then there exist $n>0$ and $0\neq f\in\mK[Y,G^{\pm1}]$ such that $fC\in AX^n \subseteq As$. Since $C=(2yx-s)g^{-1}$, this implies $fyx \in As$. But again, by the PBW basis of $A$, the element $fyx$ cannot lie in $As$. The claim follows.

	Finally, since clearly $Z(A)=A \cap Z(A_s)$, combining (\ref{Asdec}) with the claim gives
	\begin{equation*}
	Z(A)=\mK[X,Y,G^{\pm 1}] \oplus \mK[X,Y,G^{\pm 1}]C.
	\end{equation*}
	Thus $Z(A)$ is generated by $X,Y,G^{\pm1},C$ with defining relation
	$C^2=X^2G^{-1}$.

	(3) To check that $A_s$ is Azumaya, it suffices to extend scalars to an algebraically closed field. The maximal ideals of $Z(A_s)$ are $\gm_{\alpha,\beta, \gamma}=\langle C-\alpha, s^p-\beta, \,y^{2p}-\gamma \rangle$ where $\alpha, \beta \in \mK^{\times}$ and $\gamma \in \mK$. By \cite[Lemma 4.1]{Lu}, $y^{\prime 2p}=y^{2p}s^{-2p}$. Using $A_s\cong\mK[C^{\pm 1}] \otimes M_2(\CA)$, we obtain
	\begin{equation} \label{CJsm}  
	A_s/\gm_{\alpha,\beta, \gamma }A_s \cong M_2\big( \CA/(s^p-\beta, y^{\prime 2p}-\gamma\beta^{-2})\CA\big) \cong M_2(M_p(\mK)) \cong M_{2p}(\mK). 
	\end{equation}
	Thus every fibre over a maximal ideal of the centre is a full matrix algebra, and therefore $A_s$ is an Azumaya algebra of rank $(2p)^2$. 
	
	(4) Since $A_s$ is prime and Azumaya, the statement follows from the Artin--Procesi theorem (see \cite[Theorem 13.7.14]{MR}). 
	
	(5) This is a standard property of Azumaya algebras (see \cite[Proposition 13.7.9]{MR}). 
	
	(6) Since $A_s$ is Azumaya of rank $(2p)^2$,  every simple $A_s$-module has dimension $2p$. 
	
	(7) The algebra $A$ is a finitely generated module over the central subalgebra $\mK[s^p, y^{2p}, g^{\pm 2p}]$, and hence a PI ring. From the PBW basis, $A$ is in fact free of rank $8p^3$ over this subalgebra. Since $A_s$ is prime and $s$ is a regular normal element, $A$ is prime. Since $A$ and $A_s$ have the same classical  quotient ring,   $\operatorname{PI-deg}A=\operatorname{PI-deg}A_s=2p$. 
	
	(8) Since $A$ is a prime PI ring, Posner's theorem (see \cite[Theorem 13.6.5]{MR}) implies that $Q(A)$ is obtained by inverting the nonzero central elements and is central simple over $K$. 
	Since $Q(A)=Q(A_s)$, it follows from Lemma \ref{7Feb26} that  $Q(A)=Q(A_s) \cong A_s \otimes_{\mathscr{Z}} K \cong M_2(\CA) \otimes_{\mathscr{Z}} K$.  The dimension follows from the PI degree.
\end{proof}

\subsection{Prime ideals and Azumaya locus}

The following theorem describes the prime, primitive, and maximal ideals of the bosonization of the super Jordan plane  in positive characteristic.
\begin{theorem} \label{28Jan26} 
	Assume that $\mK$ is algebraically closed of characteristic $p>2$.
	\begin{enumerate}
		\item The prime spectrum of $A$ is given by
	\begin{equation*}
	\Spec(A)= \{\langle s,\,x, \,\gp \rangle \mid \gp \in \Spec(\Lambda)   \} \,\, \cup \,\, \{ \gq A_s \cap A \mid \gq \in \Spec(\mK[C^{\pm 1}, s^{\pm p}, y^{2p}]) \}. 
	\end{equation*}
	
    \item The primitive and maximal spectra of $A$ coincide and are given by
	\begin{align*}
	\Prim(A)= \Max(A)=&\{ \langle s, \, x, \, y,\, g-\gamma \mid \gamma \in \mK^{\times} \}\\ &\cup \,\, \{ \langle s,\, x,\,  y^2-\alpha, \,g^2-\beta \rangle \mid \alpha, \beta\in \mK^{\times} \}  \\ &\cup \,\, \{ \langle C-\alpha, s^p-\beta, y^{2p}-\gamma \mid \alpha, \beta \in \mK^{\times}, \gamma \in \mK \}. 
	\end{align*}
	The corresponding primitive factor algebras are as follows:
	\begin{enumerate}
\item For any $\gamma \in \mK^{\times}$, $A/\langle s, \, x,\, y, \, g-\gamma \rangle \cong \mK$. 
\item For any $\alpha, \beta \in \mK^{\times}$, $A/\langle s, \,x,\, y^2-\alpha, \,g^2-\beta \rangle \cong M_2(\mK)$. 
\item For any $\alpha,\beta \in \mK^{\times}$ and $\gamma \in \mK$, $A/\langle C-\alpha, s^p-\beta, y^{2p}-\gamma \rangle \cong M_{2p}(\mK)$. 
	\end{enumerate}
		\end{enumerate}
\end{theorem}
\begin{proof}
	(1) Let $\Spec(A, s)$ (resp. $\Spec_s(A)$) denote the set of prime ideals of $A$ that contain (resp. do not contain) $s$. Then $\Spec(A)=\Spec(A, s) \cup \Spec_s(A)$.
	
	By Lemma \ref{11Mar26}, every prime ideal containing $s$ also contains $x$. Therefore, $\Spec(A, s)$ is in bijection with the prime spectrum of  $\Lambda=A/\langle s, x \rangle$. 
	  It follows that $\Spec(A, s)=\{ \langle s, \, x,\, \gp \rangle \mid \gp \in \Spec(\Lambda) \}$.	 	
	Since $s$ is normal in $A$, the prime ideals of $A$ not containing $s$ are in bijection with the prime ideals of the localization $A_s$. By Proposition \ref{23Jan26}(5), the prime ideals of $A_s$ are precisely those extended from its centre $Z(A_s)=\mK[C^{\pm 1}, s^{\pm p}, y^{2p}]$, that is, $\Spec(A_s)=\{ \gq A_s \mid \gq \in \Spec(\mK[C^{\pm 1}, s^{\pm p}, y^{2p}]) \}$. Contracting back to $A$ gives $\Spec_s(A)=\{ \gq A_s \cap A \mid \gq \in \Spec(\mK[C^{\pm 1}, s^{\pm p}, y^{2p}]) \}$. Combining the two cases yields the desired description of $\Spec(A)$.
	
	(2) Since $A$ is a PI algebra, every primitive ideal is maximal. It follows from statement (1) that 
	\begin{align*}
	\Max(A)= \{ \langle s, \, x, \, y,\, g-\gamma \mid \gamma \in \mK^{\times} \}\,\, \cup \,\, &\{ \langle s,\, x,\,  y^2-\alpha, \,g^2-\beta \rangle \mid \alpha, \beta\in \mK^{\times} \}  \\ \cup \,\, &\{ \gm A_s \cap A \mid \gm \in \Max(\mK[C^{\pm 1}, s^{\pm p}, y^{2p}]) \}. 
	\end{align*}
	Since $\mK$ is algebraically closed, the maximal ideals of $\mK[C^{\pm 1}, s^{\pm p}, y^{2p}]$ are precisely $\gm_{\alpha,\beta, \gamma}=\langle C-\alpha, s^p-\beta, \,y^{2p}-\gamma \rangle$ where $\alpha, \beta \in \mK^{\times}$ and $\gamma \in \mK$. We claim that $\gm_{\alpha,\beta, \gamma} A_s \cap A=\gm_{\alpha,\beta, \gamma}A$. Indeed, since $\beta \in \mK^{\times}$, the image of $s$ is invertible in $A/\gm_{\alpha,\beta,\gamma}A$. Consequently,
	\begin{equation*}
	A/\gm_{\alpha,\beta,\gamma}A \cong A_s/\gm_{\alpha,\beta,\gamma}A_s. 
	\end{equation*}
	By \eqref{CJsm}, the latter algebra is isomorphic to $M_{2p}(\mK)$, which is simple. This shows that $\gm_{\alpha,\beta,\gamma}A$ is a maximal ideal of $A$ and proves the claim. The description of the primitive factor algebras follows immediately. 
\end{proof}

Let $R$ be an algebra that is finite over its centre $Z=Z(R)$. The \emph{Azumaya locus} of $R$ over $Z$ is the dense open subset of $\Max(Z)$ defined by
\begin{equation*}
\operatorname{Az}(R):=\{ \gm \in \Max(Z) \mid R_{\gm} \text{ is Azumaya over } Z_{\gm} \},
\end{equation*} 
where $R_{\gm}$ and $Z_{\gm}$ denote the localizations of $R$ and $Z$ at $\gm$, respectively. Equivalently, a maximal ideal $\gm \in \Max(Z)$ lies in $\operatorname{Az}(R)$ if and only if $R/\gm R$ is a central simple algebra over the field $Z/\gm$. On the Azumaya locus, the PI degree of $R/\gm R$ achieves the PI degree of $R$. Thus, $\gm \in \operatorname{Az}(R)$ if and only if $\gm R$ is the annihilator of a simple $R$-module of dimension $\operatorname{PI-deg} R$.

\begin{remark}
		Assume that $\mK$ is algebraically closed of characteristic $p>2$. Then the Azumaya locus of $A$ is the open subset
	\begin{align*}
	\operatorname{Az}(A)= \{\gm\in\Max (Z(A))\mid s^p \notin\gm\}.
	\end{align*}
	These are precisely the maximal ideals for which the image of $s$ is invertible in $A/\gm A$.
	Explicitly, 
	\begin{align*}
	\operatorname{Az}(A)
	=\{\gm_{\alpha,\beta,\gamma} \mid \alpha, \beta \in \mK^{\times},\, \gamma\in \mK \}, 
	\end{align*}
	where $\gm_{\alpha,\beta, \gamma}=\langle C-\alpha, s^p-\beta, \,y^{2p}-\gamma \rangle \in \Max(Z(A)).$ 
	For every $\gm \in \operatorname{Az}(A)$, we have $A/\gm A \cong M_{2p}(\mK)$, and is therefore a central simple algebra of PI degree $2p$. The complement consists of maximal ideals containing $s^p$, where the PI degree drops.
\end{remark}

\subsection{Classification of simple $A$-modules}

Suppose that ${\rm char}\,\mK=p>2$. Then $A$ is a PI algebra, and hence every simple $A$-module is finite-dimensional. Moreover, by the description of the primitive factors of $A$ in Theorem \ref{28Jan26}(2), every simple $A$-module has dimension $1$, $2$ or $2p$.

We now construct a family of simple $A$-modules $V_{\alpha,\beta,\gamma}$ of dimension $2p$, indexed by $\alpha,\beta \in \mK^{\times}$ and $\gamma \in \mK$. We use the notation $\llbracket a, b\rrbracket:=\{i\in \N \mid a \leq i \leq b \}$. 
As a vector space,
\begin{equation*}
V_{\alpha,\beta,\gamma}= V_0 \oplus V_1, \quad \text{where} \quad  V_0=\bigoplus_{i=0}^{p-1}\mK v_i, \quad V_1= \bigoplus_{i=0}^{p-1} \mK w_i.
\end{equation*}
  The $A$-action is given by 
\begin{equation} \label{Vabc}  
\begin{aligned}
x v_i &=0, \quad   i\in\llbracket 0, \,p-1\rrbracket,  &\qquad x w_i&=v_i,  \quad  i\in\llbracket 0, \,p-1\rrbracket; \\
yv_i &=\begin{cases}
w_{i+1},   &  i\in\llbracket 0, \,p-2\rrbracket, \\
\beta w_0, & i=p-1, 
\end{cases} & yw_i &=\begin{cases}
(i+1)v_i+\gamma v_{i+1},  & i\in\llbracket 0, \,p-2\rrbracket, \\
\beta \gamma v_0, & i=p-1;
\end{cases}\\
gv_i &=\begin{cases}
-\alpha^{-1} v_{i+1},  &i\in\llbracket 0, \,p-2\rrbracket, \\
-\alpha^{-1}\beta v_0, & i=p-1, 
\end{cases}
& gw_i &= \begin{cases}
\alpha^{-1} w_{i+1},  &i\in\llbracket 0, \,p-2\rrbracket, \\
\alpha^{-1}\beta w_0, & i=p-1.
\end{cases}
\end{aligned}
\end{equation}
A straightforward computation using \eqref{Vabc} shows that
\begin{equation*}
sv_i =\begin{cases}
 v_{i+1},  &i\in\llbracket 0, \,p-2\rrbracket, \\
\beta v_0, & i=p-1, 
\end{cases} \qquad 
 sw_i = \begin{cases}
 w_{i+1},   &i\in\llbracket 0, \,p-2\rrbracket, \\
\beta w_0, & i=p-1.
\end{cases}
\end{equation*}
Using these identities, one easily verifies that the defining relations of $A$
are satisfied. Therefore $V_{\alpha,\beta,\gamma}$ is an $A$-module.

\begin{proposition} \label{26Mar26}  
	Assume that $\mK$ is algebraically closed of characteristic $p>2$.
For any $\alpha,\beta \in \mK^{\times}$ and $\gamma \in \mK$, the $A$-module $V_{\alpha,\beta,\gamma}$ is simple, with annihilator
\begin{equation*}
\ann_A (V_{\alpha,\beta,\gamma})=\langle C-\alpha, s^p-\beta, y^{2p}+\beta^2\gamma^p \rangle.
\end{equation*}
 Moreover, $V_{\alpha,\beta,\gamma} \cong V_{\alpha',\beta',\gamma'}$ if and only if $\alpha'=\alpha$, $\beta'=\beta$, and $\gamma^{\prime p}=\gamma^p$. 
\end{proposition}
\begin{proof}
	Observe first that $xV_0=0$ and $xV_1=V_0$, and that $V_{\alpha,\beta,\gamma}$ is generated by $v_0$.  Let $W$ be a nonzero submodule of $V_{\alpha,\beta,\gamma}$. We claim that $W \cap V_0 \neq 0$.  Take a nonzero element $v=v_0+v_1\in W$ with $v_i\in V_i$. If $v_1=0$ there is nothing to prove. Otherwise, applying $x$ gives $xv=xv_1 \in W \cap V_0$, which is nonzero. 
	 Now take a nonzero element $v\in W \cap V_0$ and write $v=\sum_{i=0}^{m}\mu_i v_i$ with $\mu_m\in \mK^{\times}$. 
	 	 Set $\partial:=\alpha^{-2}(yg^{-1})^2+\gamma$. Using the defining formulas \eqref{Vabc}, we obtain
	 \begin{equation*}
	 \partial v_0=0, \qquad  \partial v_i=-i v_{i-1} \quad \text{for } i\in \llbracket 1, \,p-1\rrbracket. 
	 \end{equation*}
	 It follows that $\partial^m v =(-1)^m\mu_m  m! v_0 \in W$, so $v_0 \in W$. Since $v_0$ generates $V_{\alpha,\beta,\gamma}$, we conclude that $W = V_{\alpha,\beta,\gamma}$, proving simplicity.
	 	 
    For the annihilator, recall from \cite[Remark 4.2]{Lu} that
    $\prod_{i=0}^{2p-1}(y+ix)=y^{2p}.$ 
Using this together with $g^{-1}y=-(y+x)g^{-1}$, we obtain
\begin{equation} \label{ygp}  
(yg^{-1})^{2p}=\prod_{i=0}^{2p-1}(y+ix)\cdot g^{-2p}=y^{2p} g^{-2p}. 
\end{equation}
Since $(yg^{-1})^2 v_0=-\alpha^2 \gamma v_0$, we have $(yg^{-1})^{2p}v_0=-\alpha^{2p}\gamma^p v_0$. On the other hand, by \eqref{ygp},
$$(yg^{-1})^{2p}v_0=y^{2p} g^{-2p}v_0=\alpha^{2p}\beta^{-2} y^{2p} v_0.$$
Comparing these expressions gives $y^{2p}v_0=-\beta^2\gamma^p v_0$. 
A straightforward verification also shows $Cv_0=\alpha v_0$ and $s^p v_0= \beta v_0$. Since $C$, $s^p$ and $y^{2p}$ are central and $v_0$ generates $V_{\alpha,\beta,\gamma}$, we obtain
\begin{equation*}
\ga:=\langle C-\alpha, s^p-\beta, y^{2p}+\beta^2\gamma^p \rangle \subseteq \ann_A (V_{\alpha,\beta,\gamma}).
\end{equation*}
 By Theorem \ref{28Jan26}(2), $\ga$ is a maximal ideal. Hence, $\ann_A (V_{\alpha,\beta,\gamma})=\ga$. 

Finally, since $A/\ga \cong M_{2p}(\mK)$ has a unique simple module, two modules $V_{\alpha,\beta,\gamma}$ and $V_{\alpha',\beta',\gamma'}$ are isomorphic if and only if their annihilators coincide, i.e., $\alpha'=\alpha$, $\beta'=\beta$, and $\gamma^{\prime p}=\gamma^p$. 
\end{proof}

The following theorem classifies all simple $A$-modules. 
\begin{theorem} \label{25Mar26}  
	Let $\mK$ be algebraically closed of characteristic $p>2$. 
 Then the set of isomorphism classes of simple $A$-modules is 
		\begin{equation*}
		\widehat{A}=\{ [S_{\gamma}] \mid \gamma \in \mK^{\times} \} \,\,\cup \,\, \{ [T_{\alpha,\beta}] \mid \alpha,\beta \in \mK^{\times} \}\,\cup \, \{ [V_{\alpha, \beta,\gamma}] \mid \alpha,\beta \in \mK^{\times}, \gamma \in \mK \},  
		\end{equation*}
		where  $S_{\gamma}:=A/\langle s, \, x, \, y,\, g-\gamma \rangle$,  the modules $T_{\alpha,\beta}$ are defined in \eqref{Tab}, and the modules $V_{\alpha,\beta,\gamma}$ are defined in \eqref{Vabc}.
\end{theorem}
\begin{proof}
	Let $M$ be a simple $A$-module. Since $s$ is normal, either $sM=0$ or $s$ acts bijectively on $M$. If $sM=0$, then by Theorem \ref{28Jan26}(2) the annihilator of $M$ is either $\langle s, \,x,\, y, \, g-\gamma\rangle$ for some $\gamma\in \mK^{\times}$ or $\langle s, x, y^{2}-\alpha^2, g^2-\beta \rangle$ for some $\alpha,\beta \in \mK^{\times}$. In the first case $M \cong S_{\gamma}$, while in the second case $M \cong T_{\alpha,\beta}$. Suppose now that $s$ acts bijectively on $M$. Let $P := \ann_A(M)$, which is a primitive ideal. By Theorem \ref{28Jan26}(2), we have $P=\langle C-\alpha, s^p-\beta, y^{2p}-\d \rangle $ for some $\alpha,\beta \in \mK^{\times}$ and $\d \in \mK$.  Since $\mK$ is algebraically closed, there exists $\gamma\in \mK$ such that $\gamma^p=-\d \beta^{-2}$.  By Proposition \ref{26Mar26} we then have $\ann_A (V_{\alpha,\beta,\gamma})=P$.  Thus both $M$ and $V_{\alpha,\beta,\gamma}$ are simple modules over the factor algebra $A/P$. Since $A/P \cong M_{2p}(\mK)$ has a unique simple module up to isomorphism, it follows that $M \cong V_{\alpha,\beta,\gamma}$. 	This completes the proof.
\end{proof}





\small{

\end{document}